\documentclass[11pt]{amsart}
\usepackage{amscd, amssymb, mathrsfs, mathabx, 
dsfont, xcolor, %mathtools, 
amsmath}
\usepackage{hyperref}   
\usepackage{collectbox}



\input{xy}
\xyoption{all}

\newtheorem{Thm}{Theorem}[section]

\newtheorem{Prop}[Thm]{Proposition}
\theoremstyle{definition}
\newtheorem{Def}[Thm]{Definition}
\newtheorem{Def/Thm}[Thm]{Definition/Theorem}
\newtheorem{Cor}[Thm]{Corollary}
\newtheorem{lemma}[Thm]{Lemma}

\theoremstyle{remark}
\newtheorem{Rmk}[Thm]{Remark}

\numberwithin{equation}{section}
\newcommand{\ti }{\times}
\newcommand{\ot }{\otimes}
\newcommand{\ra }{\rightarrow}

\newcommand{\Hom }{{\mathrm{Hom}}}
\newcommand{\tr }{{\mathrm{tr}}}

\newcommand{\Spec}{{\mathrm{Spec}}}

\newcommand{\cA}{{\mathcal{A}}}
\newcommand{\cO}{{\mathcal{O}}}

\newcommand{\cQ}{{\mathcal{Q}}}

\newcommand{\HH}{{\mathbb H}}

\newcommand{\PP }{{\mathbb P}}
\newcommand{\GG }{{\mathbb G}}

\newcommand{\CC }{{\mathbb C}}
\newcommand{\ZZ }{{\mathbb Z}}
\newcommand{\RR }{{\mathbb R}}

\newcommand{\kb }{{\beta}}
\newcommand{\ka }{{\alpha}}

\newcommand{\ch}{\mathrm{ch}}

\newcommand{\lan}{\langle}
\newcommand{\ran}{\rangle}

\newcommand{\LL}{\mathbb{L}}

\newcommand{\td}{\mathrm{td}}
\newcommand{\Ch}{\mathrm{Ch}}
\newcommand{\AAA}{\mathbb{A}}

\newcommand{\cB}{\mathcal{B}}

\newcommand{\End}{\mathrm{End}}

\newcommand{\str}{\mathrm{str}}

\newcommand{\Perf}{\mathrm{Perf}}

\newcommand{\Com}{\mathrm{Com}}

\newcommand{\vC}{\text{\rm \v{C}}}

\newcommand{\Om}{\Omega ^{\bullet}}

\newcommand{\fU}{\mathfrak{U}}

\newcommand{\MFdg}{D_{dg}}

\newcommand{\Kun}{\text{K\" unn}}

\begin{document}
\title{Hirzebruch-Riemann-Roch for global matrix factorizations}  

\begin{abstract}We prove a Hirzebruch-Riemann-Roch type formula for global matrix factorizations.
This is established by an explicit realization of the abstract Hirzebruch-Riemann-Roch type formula 
of Shklarov. We also show a Grothendieck-Riemann-Roch type theorem.
\end{abstract}

\author[B. Kim]{Bumsig Kim}
\address{Korea Institute for Advanced Study\\
85 Hoegiro, Dongdaemun-gu \\
Seoul 02455\\
Republic of Korea}
\email{bumsig@kias.re.kr }
%%\address{KIAS\\
%%Seoul\\
%%Korea}

%\date{\today}

\thanks{The work is supported by KIAS individual grant MG016404.}



\maketitle

%\tableofcontents

\section{Introduction}\label{sec: intro}
Let $k$ be a field of characteristic zero and let
$\GG$ be either the group $\ZZ$ or $\ZZ/ 2$.
We consider  a $\GG$-graded dg enhancement  $\MFdg (X, w)$ of the derived category  of matrix factorizations for $(X, w)$.
Here $X$ is an $n$-dimensional nonsingular variety over $k$ and $w$ is a regular function on $X$.
An object of $\MFdg (X, w)$ is a $\GG$-graded vector bundle $E$ on  $X$ equipped with
a degree $1$, $\cO_X$-linear homomorphism $\delta _E : E \to E$ such that $\delta _E ^2 = w \cdot \mathrm{id}_E$.
The structure sheaf $\cO_X$ is by definition $\GG$-graded but concentrated in degree $0$.
The degree of $w$ is 2. If $w$ is nonzero, then $\GG$ is forced to be $\ZZ/2$.
Assume that the critical locus of $w$ is set-theoretically in $w^{-1}(0)$ and proper over $k$.

The Hochschild homology $HH _* (\MFdg (X, w) )$  of $\MFdg (X, w)$ is naturally isomorphic to \[\HH ^{-*} (X, (\Omega ^{\bullet} _{X}, -dw )) ; \]
see \cite{LP, Platt}.
The isomorphism is called the Hochschild-Kostant-Rosenberg (in short HKR) type isomorphism, denoted by $I_{HKR}$.
Here  $(\Omega ^{\bullet} _X, -dw )$ is a $\GG$-graded complex $\bigoplus _{p\in \ZZ} \Omega_X ^p [p]$ with the differential $-dw\wedge$.
Let $\ch (E) \in \HH ^{0} (X, (\Omega ^{\bullet} _X, -dw )) $ be
the image of the categorical Chern character  $\Ch (E) \in HH _0 (\MFdg (X, w) )$ of $E$ under $I_{HKR}$. 
In this paper we prove the following Hirzebruch-Riemann-Roch type formula.

\begin{Thm}\label{thm: main} For matrix factorizations $P$ and  $Q$ in $\MFdg (X, w)$ we have
\begin{equation}\label{eqn: main}     \sum_{i \in \GG}  (-1)^i \dim  \RR ^i \Hom (P, Q)   
    =   (-1)^{ \binom{n+1}{2}}  \int_X \ch (P) ^{\vee} \wedge \ch (Q) \wedge \td (X)  ,  \end{equation}
where $\td (X) \in \oplus _{p\in \ZZ} H ^0 (X, \Omega ^p _{X} [p])$ is the Todd class of $X$.
\end{Thm}

We explain notation in the above theorem.
Firstly, the operation $\wedge$ is the wedge product inducing 
\begin{equation} \label{eqn: wedge}  \HH ^{*} (X, (\Omega ^{\bullet} _X, -dw )) \ot \HH ^{*} (X, (\Omega ^{\bullet} _X, dw )) 
\ot  \HH^{*}(X, (\Omega ^{\bullet} _X, 0)) \xrightarrow{wedge }  \oplus _p H^{*} _c(X, \Omega ^p _X[p]) ; \end{equation}
see  \ref{def: Wedge} for details.
Secondly, $ \int_X$ is the composition 
\begin{equation} \label{eqn: theta}\oplus _{p\in \ZZ} H^{*} _c(X, \Omega ^p _X[p]) \xrightarrow{proj} 
 H^0_c (X, \Omega ^n _X [n]) \xrightarrow{\tr _X} k \end{equation} 
  of the projection and  the canonical trace map $\tr _X$ for the properly supported cohomology; see \S~\ref{def int}. Thirdly,  $\vee$ is induced from 
  a chain map
\[ (\Omega ^{\bullet}_X,  dw ) \to (\Omega ^{\bullet}_X, -dw), \text{ defined by } (-1)^p\mathrm{id}: \Omega ^p _X \to \Omega ^p _X \]
in each component.

For a proper dg category $\cA$ there is an abstract Hirzebruch-Riemann-Roch formula \eqref{eqn: abs HRR} due to Shklyarov \cite{Shk: HRR}. 
By an explicit  realization of the formula for $\cA = \MFdg (X, w) $ we will obtain Theorem~\ref{thm: main}.
Let \[\lan , \ran _{can} : HH_* (\cA) \ot HH_* (\cA^{op}) \to k \]  be
the so-called  canonical pairing for $\cA$. Here $\cA ^{op}$ is the opposite category of $\cA$ and there is
an isomorphism  $\cA ^{op} \cong \MFdg (X,- w)$; see  \S~\ref{sub: geom diag}. This yields
\begin{equation}\label{diag: vv} \xymatrix{   HH_* (\cA ) \ar[r]_(.4){I_{HKR}} \ar[d]_{\vee} &  \HH ^{-*} (X, (\Om _X , - dw )) \ar[d]_{\vee}  \\
               HH_* (\cA ^{op}) \ar[r]_(.4){I_{HKR}} &  \HH ^{-*} (X, (\Om _X ,  dw )) , } \end{equation}
               where the left $\vee$ is defined to make the diagram commute. 
Shklyarov's formula says that the left-hand side of \eqref{eqn: main}
is equal to $\lan \Ch (Q), \Ch (P)^{\vee} \ran_{can}$. 
% where $\vee$ is the corresponding dual map in Hochschild homology (see \S~\ref{sub: dual}).
  Therefore Theorem~\ref{thm: main} is reduced to an explicit realization
of the pairing.  % under the commuting diagram of isomorphisms

               \begin{Thm}\label{thm: exp pairing} 
The canonical pairing $\lan , \ran_{can}$ under $I_{HKR}$
%\[ HH_* (\cA ) \ot HH_* (\cA ^{op} ) \to  \HH^{-*} (\Omega ^{\bullet}_X, -dw) \ot \HH ^{-*} (\Omega ^{\bullet}_X, dw)  \]
corresponds to \[  (-1)^{\binom{n+1}{2}}  \int _X (\cdot \wedge \cdot \wedge \td (X)) . \]
\end{Thm}

       The following formula for $\ch (E)$ is established in 
        \cite{CKK, KP, Platt}. Let $\fU = \{ U_i \}_{i \in I}$ be an affine open covering of $X$ and let $\nabla _i$ be a connection of $E|_{U_i}$.
In the \v{C}ech hypercohomology  $\check{\HH} ^* (\fU, (\Om _X, (-1)^i dw\wedge ))$
\[ \ch (E) = \str \exp (-([\nabla _i, \delta _E])_{i\in I} - (\nabla _i - \nabla _j)_{i, j \in I, i < j }) . \]
Here $\str$ means the supertrace and the products in the exponential are  Alexander-\v{C}ech-Whitney cup products in the \v{C}ech complex
$\vC ^* (\fU, (\End (E) \ot \Om _X))$; see \cite{CKK} for details.

Remarks on others' related works are in order.
When $\GG = \ZZ$ and $X$ is a projective variety over $\CC$, 
there is a natural isomorphism between Hodge cohomology and the singular (or equivalently $C^{\infty}$-de Rham) cohomology of the associated complex manifold $X^{an}$:   
$\oplus _{p, q} H^q (X, \Omega ^p_X[p]) \overset{\phi}{\ra}  H^{-\bullet} (X^{an}, \CC)$.
Let $\mathrm{tw}$ be an automorphism of  $\oplus _{p, q} H^q (X, \Omega ^p_X[p]) $ sending 
a $(p, q)$-form $\gamma ^{p, q}$ to $(\frac{1}{2\pi i })^{p} \gamma ^{p, q}$, then
$\phi (\mathrm{tw} ( \ch (E)))$ coincides with the topological Chern character  $\ch _{top} (E)$ of $E$. 
The right-hand side of \eqref{eqn: main} becomes \[ \int _{X^{an}} \ch _{top }(P^{\vee}) \ch _{top} (Q) \td _{top}(X) . \]
Here $\int _{X^{an}}$ denotes the usual integration and $\td _{top} (X)$ is the usual Todd class of $X^{an}$.
Hence Theorem \ref{thm: main} is the usual Hirzebruch-Riemann-Roch theorem \cite{Hir}.
When $\GG = \ZZ$ and $k = \CC$, Theorem \ref{thm: main} is the O'Brian -- Toledo -- Tong  theorem for algebraic coherent sheaves \cite{OTT}.
When $\GG = \ZZ$, Theorem  \ref{thm: main} and its generalization Corollary \ref{cor: gen}  coincide with Theorem 4 of Markarian \cite{Mark} and 
some works of C\u{a}ld\u{a}raru -- Willerton \cite{CW} and Ramadoss \cite{Ram: HRR},  respectively.

When $\GG = \ZZ /2$,  $X$ is an open subscheme of $\AAA ^n _k $ containing the origin, and $w$ has only one singular point at the origin, 
 the composition of wedge products \eqref{eqn: wedge} and $\int _X$ in \eqref{eqn: theta} 
is a residue pairing as shown in \cite[Proposition 4.34]{BW: ShkConj}.
Thus  in this case,  Theorem  \ref{thm: main} is  the Polishchuk -- Vaintrob theorem \cite[Theorem 4.1.4]{PV: HRR}
and Theorem \ref{thm: exp pairing} is  the Brown -- Walker theorem \cite[Theorem 1.8]{BW: ShkConj} 
proving a conjecture of Shklyarov \cite[Conjecture 3]{Shk: Residue}. 

It is natural to consider the stacky version of Theorem \ref{thm: main}. It will be treated elsewhere \cite{CKS}.

\medskip

\noindent{\em Conventions}: Unless otherwise stated a dg category is meant to be a $\GG$-graded dg category over $k$.
For a variety $X$ over $k$, we write simply  $\Omega ^p_{X}$ for the sheaf $\Omega ^p_{X/k}$ of relative differential $p$-forms of $X$ over $k$.
For a homogenous element $a$ in a  $\GG$-graded $k$-space, $|a|$ denotes the degree of $a$. 
 For a dg category $\cA$ we often write  $\cA (x, y)$  instead of the Hom complex $\Hom _{\cA} (x, y) $ between objects $x, y$ of $\cA$. 
 We write $x \in \cA$ if $x$ is an object of $\cA$.
For a dg algebra $A$, $C(A)$ denotes the Hochschild  $\GG$-graded complex $\bigoplus _{n\ge 0} A \ot A[1]^{\ot n}$ with differential $b$.
Similarly for a dg category $\cA$, $C(\cA)$ denotes the Hochschild complex of $\cA$; see for example \cite{BW: ShkConj, Shk: HRR}.
For $a_i \in \cA (x_{i+1}, x_i) $, $i=1, ..., n$, $x_i \in \cA$, we  write  $a_0[a_1| ... | a_n] $ for $a_0 \ot sa_1 \ot ... \ot sa_n$,
where $s$ is the suspension so that $|sa| = |a| -1$. The symbol $\str$ stands for the supertrace.
By a coherent factorization for $(X, w)$ we mean a $\GG$-graded coherent $\cO_X$-sheaf $E$ with a curved differential $\delta _E$ such that
$\delta _E ^2 = w \cdot \mathrm{id}_E$.
%a coherent factorization ... 

\medskip
 
 \noindent{\em Acknowledgements}:  The author thanks David Favero, Taejung Kim, and Kuerak Chung for
 useful discussions. 

\section{Abstract Hirzebruch-Riemann-Roch}

%\subsection{Known Results}
Following mainly \cite{PV: HRR, Shk: HRR} we review the abstract Hirzebruch-Riemann-Roch theorem
in the framework of Hochschild homology theory.

\subsection{Categorical Chern characters}

For a $\GG$-graded dg category $\cA$ over $k$ let $C (\cA)$ be the Hochschild  complex of $\cA$.
For $x \in \cA$, the identity morphism $1_x$ of $x$ is a $0$-cycle element and hence it defines a class 
\[ \Ch (x) := [1_x] \in HH _0 ( \cA ) := H ^0 (  C (\cA) ) , \]
which is called the categorical Chern character of $x$. 
For an object $y$ of the dg category $ \Perf \cA$ of perfect right $\cA$-modules, we also regard $\Ch (y)$ as
an element of $HH _0 (\cA)$ by the canonical isomorphism $HH_*(\Perf \cA) \cong HH_* (\cA) $. 

\subsection{K\" unneth isomorphism}

Let $\cA$, $\cB$ be dg categories.
We define a natural chain map over $k$
\begin{align*}  C  (\cA) \ot C (\cB) & \to   C (\cA \ot \cB) \\
                              a_0[a_1| ... | a_n] \ot b_0 [b_1| ... | b_m] & \mapsto   (-1)^{|b_0| ( \sum_{i=1}^n (|a_i| -1))} a_0 \ot b_0     \sum_{\sigma} {\pm} 
                              [c_{\sigma (1)} |  ... | c_{\sigma (n+m)} ] ,
 \end{align*}
 where $\sigma$ runs for all $(n, m)$-shuffles,  $c_1 = a_1 \ot 1, ..., c_n = a_n \ot 1, c_{n+1} = 1\ot b_1, ..., c_{n+m} = 1 \ot b_{m}$, and
 the rule of sign $\pm$ is determined by the Koszul sign rule.
 Here after each shuffle, each $1$ is uniquely replaced by an appropriate identity morphism so that
 the outcomes make sense as elements of the Hochschild complex $C (\cA \ot \cB) $.
The Eilenberg-Zilber theorem says that the chain map is a quasi-isomorphism.
We call the induced isomorphism 
\[ HH_* (\cA) \ot HH_* (\cB) \to HH_* (\cA \ot \cB) \]
the K\" unneth isomorphism, denoted by $\Kun$.

\subsection{The diagonal bimodule}
 We denote by $\cA ^{op}$ the opposite category of $\cA$. For $x\in \cA$ we write
$x^{\vee}$ for the object of $\cA ^{op}$ corresponding to $x$.
Let  $\Com _{dg} k$ be the dg category of complexes over $k$.
The diagonal $\cA$-$\cA$-bimodule $\Delta _{\cA}$ of  a dg category $\cA$ is defined to
be the dg functor  \[ \Delta _{\cA} : \cA  \ot \cA ^{op}  \to \Com _{dg} k ; \ \  y \ot x ^{\vee} \mapsto \Hom_{\cA} (x, y) . \]
Since $\cA \ot \cA ^{op} \cong ( \cA ^{op}  \ot \cA )^{op}$,  $\Delta _A$ is a right $\cA ^{op} \ot \cA $-module.
Assume that  $\cA$ is proper, i.e., 
\[\sum _{i \in \GG} \dim H^i (\Hom _{\cA} (x, y))  < \infty \] for all $x, y \in \cA$.
Then we may replace the codomain of $\Delta _{\cA}$ by  the dg category $\Perf k $ of perfect dg $k$-modules.

\subsection{The canonical pairing} For a proper dg category $\cA$,
the canonical pairing  $\lan\ ,\ \ran _{can} $ is defined as the composition 
\begin{equation*} HH_*(\cA) \ti HH_* (\cA ^{op} ) \xrightarrow{\Kun}  HH_* (\cA \ot \cA ^{op}) \xrightarrow{\Delta _*}  HH_* (\Perf k) \cong k, \end{equation*} 
where $\Delta _*$ is the homomorphism in Hochschild homology level induced from the dg functor $\Delta _{\cA}$. Here we use
the canonical isomorphism $ HH_*(\Perf k)  \cong k $ making a commuting diagram for $C \in \Perf k$
\begin{align}\label{eqn: perf k tr} \xymatrix{  HH_*(\Perf k) & \ar[l]_{\cong} HH_* (k)  \cong  k , \\ 
                                                 Z\End (C) \ar[u]^{natural} \ar[ru]_{\str}  &   }  \end{align}
                                                 where $Z\End (C)$ is the graded $k$-space of closed endomorphisms of $C$ and $\str$ denotes the supertrace. 
Since $HH_*(\Perf k) = HH_0 ( \Perf k)$, the pair $\lan \gamma, \gamma ' \ran$ for $\gamma  \in HH_{p} (\cA)$, $\gamma '  \in HH_{p'} (\cA ^{op})$ 
can be nontrivial only when $p+ p' = 0$.

\subsection{A proposition}
Let $\cA$ be a proper dg category. 
Let $M$ be a perfect right $\cA ^{op} \ot \cB$ module, in other words, a perfect $\cA$-$\cB$-bimodule.
Denote by $T_M$ the dg functor \[ \ot _{\cA} M : \Perf \cA  \to \Perf \cB  \ \text{ sending } N \mapsto N \ot _{\cA}  M \]
and denote the induced map in Hochschild homology by \[ (T_M)_*: HH _* (\cA) \to HH_* (\cB) . \]

\begin{Prop} \label{prop: char can} \cite[Proposition 4.2]{Shk: HRR} 
If we write $\Ch (M) = \sum _i  t_i \ot t^i \in HH_* (\cA ^{op} ) \ot HH_* (\cB ) \cong HH_* (\cA ^{op} \ot \cB) \cong \HH _* (\Perf (\cA ^{op} \ot \cB ))$ via the K\"unneth isomorphism and the canonical isomorphism, then for every $\gamma \in HH _p  (\cA )$ we have
\[ (T_M)_* (\gamma ) = \sum _i \lan \gamma , t_i \ran _{can} \, t^i  \in HH _p (\cB). \]
\end{Prop}

\begin{proof} The proof given in \cite{Shk: HRR} also works  for dg categories.  \end{proof}

Furthermore assume that $\cA$ is smooth, i.e, the diagonal bimodule $\Delta_{\cA}$ is perfect. Then the Hochschild homology of $\cA$ is  finite 
dimensional and hence Proposition \ref{prop: char can} can be rewritten as a commuting diagram 
\begin{equation*} \xymatrix{ HH_* (\cA ) \ar[r]^{\lan , \ran _{can}} \ar[rd]_{(T_M)_*}& HH_* (\cA ^{op}) ^* \ar[d]^{\Ch (M)} \\
             & HH_* (\cB) . }
\end{equation*}
Since $T_{\Delta _{\cA}} = \mathrm{id}_{\cA}$, the above diagram for $M = \Delta _{\cA}$ shows that $\lan, \ran _{can}$  is non-degenerate
and the canonical pairing is characterized as follows. 

Since $\Delta _{\cA} \in \Perf (\cA ^{op} \ot \cA)$,  via the K\"unneth isomorphism we can write 
 \[\Ch (\Delta _{\cA}) = \sum_i T^i \ot T_i, \text{ for some } T^i \in HH_* (\cA ^{op} )  , \ T_i \in HH_* (\cA ) . \] 
 Then %Proposition \ref{prop: char can}  in particular implies that 
$\lan , \ran _{can}$ is a unique nondegenerate $k$-bilinear map $\lan , \ran : HH_*(\cA) \ti HH_* (\cA ^{op} ) \to k $ satisfying 
\begin{equation}\label{eqn: char pairing}   
\sum _i \lan \gamma , T^i \ran \lan T_i, \gamma ' \ran = \lan \gamma, \gamma ' \ran , \text{ for every } \gamma \in HH_* (\cA ),  \gamma '\in  HH _* (\cA ^{op}) .
\end{equation} % (see \cite[Proposition 4.2]{Shk: HRR}).

\subsection{The chain map $\vee$}\label{sub: dual} Define an isomorphism of complexes 
\begin{align*} \vee : (C (\cA ), b) & \to  (C (\cA ^{op}), b)   \\
 a_0[a_1| ... | a_n] & \mapsto  (-1)^{ n + \sum_{1\le i < j \le n} (|a_i|-1)(|a_j|-1) } a_0 [a_n | ... | a_1 ] .\end{align*}

\begin{Rmk} Using the quasi-Yoneda embedding and the HKR-type isomorphism 
it is straightforward to check that the chain map $\vee$ in \S~\ref{sub: dual} fits in diagram~\eqref{diag: vv}; see for example \cite{CKK}.
\end{Rmk}

\subsection{Abstract generalized HRR}
For a proper dg category $\cA$ we may consider a sequence of natural maps 
\begin{equation}\label{eqn: diag kun}
 \xymatrix{ HH_* (\cA) \ot HH_* (\cA) \ar[r]_{\mathrm{id} \ot \vee \ \ }^{\cong \ \ \ } & HH_* (\cA) \ot HH_* (\cA ^{op}) \ar[d]_{\Kun}^{\cong}  \ar[rd]^{\lan , \ran _{can}} & \\ 
                            &       HH_* (\cA \ot \cA ^{op})   \ar[r]_{\Delta _*} &   HH_* (\Perf k  ) \cong k     . } \end{equation}

For two closed endomorphisms  $b \in \End_{\cA} (y)$, $a \in \End_{\cA}  (x)$, we
define an endomorphism $L_b \circ R_a$ of $\Hom_{\cA} (x, y)$ by sending $c \in \Hom_{\cA} (x, y)$ to $(-1)^{|a||c|} b \circ c \circ a$.
%Here  $L_b$, $R_a$ stand for the left,  right composition by $b$, $a$, respectively. 
Note that  $\Delta _* (b\ot a) = L_b \circ R_a$. Hence  from \eqref{eqn: diag kun} and  \eqref{eqn: perf k tr} we obtain
\begin{equation}\label{eqn: abs general HRR} \str (L_b \circ R_a ) =  \lan [b], [a] ^{\vee}  \ran_{can} , \end{equation}
Here $[b], [a]$ are the homology classes  in $HH_0 (\cA)$ represented by $b, a$, respectively.

\subsection{Abstract HRR}
When $a= 1_x$, $b=1_y$, \eqref{eqn: abs general HRR} yields the abstract Hirzebruch-Riemann-Roch theorem \cite{PV: HRR, Shk: HRR} for Hochschild homology:
\begin{equation}\label{eqn: abs HRR} \sum_{i\in \GG} (-1)^i \dim  H^i (\Hom_{\cA} (x, y)) ) =  \lan \Ch (y), \Ch (x)^{\vee}  \ran_{can}  . \end{equation}
This tautological HRR theorem can be useful when one expresses the right-hand side of \eqref{eqn: abs HRR} in an explicit form.

\section{Proofs of Theorems}

In this section we prove Theorems \ref{thm: main} and  \ref{thm: exp pairing}. As in \S~\ref{sec: intro}
let  $X$ be an $n$-dimensional nonsingular variety over $k$ and  $w$ is a  function on $X$ such that 
the critical locus of $w$ is in $w^{-1}(0)$ and proper over $k$.

\subsection{A geometric realization of $\Delta _{\cA}$}\label{sub: geom diag}
Let $\cA = \MFdg (X, w) $. It is proper and smooth.
There is the duality functor  \[ D: \cA^{op} \to \MFdg (X, -w) ; (E, \delta _E) \mapsto (\Hom _{\cO_X} (E , \cO _X) , \delta _E ^{\vee}) , \]
which is an isomorphism. Hence we have the HKR type isomorphism \[ HH_* (\cA ^{op} ) \cong   \HH^{-*} (\Omega ^{\bullet}_X, dw) . \]

Let $X'$ be another nonsingular variety with a global function $w'$. Assume that the critical locus of $w'$ is proper over $k$ and located on the zero locus of $w'$.
Let $\cB$ denote $\MFdg (X', w')$.
Let $\widetilde{w} := w\ot 1 - 1 \ot w' $ a global function on $Z:=X\ti X'$. %Let $\mathrm{Mod} (\cA ^{op} \ot \cB)$ denote the dg category of right $\cA ^{op} \ot \cB$-modules.
 We consider a dg functor \[ \Psi :   \MFdg (Z, -\widetilde{w}) \to \Perf (\cA ^{op} \ot \cB) \]  defined by letting
\begin{align*} \Psi (z) : \cA \ot \cB ^{op} & \to \Com _{dg} (k) ; \\
 y\ot x ^{\vee} & \mapsto   \Hom _{\MFdg (Z, -\widetilde{w})} (D(y)  \boxtimes x, z) \end{align*}
  for $y \in \cA$, $x \in \cB$, $z \in \MFdg (Z, -\widetilde{w})$. Since   $\cA ^{op} \ot \cB $ are saturated by \cite{LP},
 we may apply    Proposition 3.4 of \cite{Shk: Serre} to see that $\Psi (z)$ is indeed a perfect right $\cA ^{op} \ot \cB $-module.

Let $f: X \to X'$ be a proper morphism such that $f^* w' = w$. Then there is a dg functor 
\[ \RR f_* : \MFdg (X, w) \to \MFdg (X', w') \]  by derived pushforward; see \cite[\S~2.2]{CFGKS}.  
Define  \[ \Delta _{\RR f_*} : \cA  \ot \cB ^{op} \to \Com _{dg} k ; \ y \ot x^{\vee} \to \Hom _{\cB} (x, \RR f_* y) . \]
Again by Proposition 3.4 of \cite{Shk: Serre}, we see that  $\Delta _{\RR f_*} $ is a perfect right $\cA ^{op} \ot \cB$-module.
Let $\Gamma _f \subset X \ti X'$ denote the graph of $f$. 
Since $Z = X \ti X'$ is nonsingular, there is an object  $\cO ^{\widetilde{w}}_{\Gamma _f }$ in $\MFdg (Z, -\widetilde{w})$ which is quasi-isomorphic to
the coherent  factorization $\cO _{\Gamma _f}$ for $(Z, \widetilde{w})$.
Since \[  \Delta _{\RR f_*} ( y \ot x ^{\vee} ) =  \cB (x, \RR f_* y ) \underset{qiso}{\simeq} \Hom_{\MFdg (Z, -\widetilde{w})} (D(y) \boxtimes x,  \cO _{\Gamma _f} ^{\tilde{w}} ) , \]
 by the projection formula \cite[\S~2.2]{CFGKS}, 
 $\Delta _{\RR f_*} $ and $\Psi (\cO ^{\widetilde{w}}_{\Gamma _f }) $ are isomorphic in the derived category of right $\cA ^{op} \ot \cB$-modules.
Hence $\Ch ( \Delta _{\RR f_*} ) = \Ch (\Psi (\cO ^{\widetilde{w}}_{\Gamma _f })) $.

Consider a dg functor 
\[ \boxtimes:  \cA ^{op} \ot \cB \to  \MFdg (Z, -\widetilde{w}) ; u^{\vee} \ot v \mapsto D(u) \boxtimes v . \] 
The following commutative diagram of natural isomorphisms transforms the abstract terms to the concrete terms:
{\tiny \begin{equation}\label{diag: big comm}  \xymatrix{      
  HH_* (\cA ^{op} ) \ot HH_* (\cB)  \ar[dd]_{I_{HKR}}^{\cong}   \ar[r]_{\Kun}^{\cong}   &   HH_* ( \cA ^{op}  \ot \cB ) \ar[d]_{\boxtimes} \ar[rd]^{Yoneda}_{\cong} \\
                         & HH_* (   \MFdg (Z, -\widetilde{w})) \ar[d]_{I_{HKR}}^{\cong} \ar[d] \ar[r]_{\Psi} & HH_* (\Perf (\cA ^{op} \ot \cB)) \\
  \HH^{-*} (\Omega ^{\bullet}_X, dw) \ot \HH ^{-*} (\Omega ^{\bullet}_{X'}, -dw) 
  \ar[r]^(.6){\cong}_(.6){\text{\em K\"unneth}}  & \HH^{-*} (\Omega _{Z}^{\bullet}, d\widetilde{w}) .  &  } \end{equation}}
The commutativity of the triangle is straightforward. The commutativity of the rectangle can be seen as follows. Using the Mayer-Vietoris sequence argument,
we reduce it to the case when $X$ and $X'$ are affine. We further reduce it to the curved smooth algebra case. In the curved smooth algebra case, the commutativity 
of a corresponding diagram for Hochschild complexes of the second kind is straightforward; see for example \cite{CKK}.

We conclude  that  
\begin{equation}\label{eqn: ch Rf}  \ch ( \Delta _{\RR f_*}  )  = \ch (\cO ^{\widetilde{w}}_{\Gamma _f}) \in  \HH ^{0} (\Omega ^{\bullet}_{Z} , d\widetilde{w})  \end{equation}
by the compatibility of the K\"unneth isomorphisms and the HKR type isomorphisms in \eqref{diag: big comm}.
 In particular for $f=\mathrm{id}_X$ we have
 \begin{equation}\label{eqn: ch diag}  \ch ( \Delta _{\cA}  )  = \ch (\cO ^{\widetilde{w}}_{\Delta _X}) \in  \HH ^{0} (\Omega ^{\bullet}_{X^2} , d\widetilde{w})  \end{equation}
 if the subscript $\Delta _X$ denote $\Gamma _{\mathrm{id}_X}$.

\subsection{Some definitions}

\begin{Def}\label{def: Td}   Considering a vector bundle $F$ as an object in the derived category of coherent sheaves on $X$, we have 
the categorical Chern character of $F$ and hence $\ch (F) \in \oplus_p H ^0 (X, \Omega ^p_X [p]) \underset{I_{HKR}}{\cong} HH _0 (D^b (\mathrm{coh} (X)))$. Using this 
and the Todd class formula in terms of Chern roots we define 
$\td (F) \in \bigoplus _{p} H^0 (X, \Omega ^p_X[p])$, which we  call the {\em Todd class} of $F$ valued in Hodge cohomology. 
We write $\td (X)$ for $\td (T_X)$, called the Todd class of $X$. Similarly, we define the $i$-th Chern class $c_i (F)$ of $F$ valued in Hodge cohomology.
\end{Def}

\begin{Def}\label{def: Wedge}  Let $w_i \in \Gamma (X, \cO _X)$, $i=1, 2$ and let $Z_i$ be the critical locus of $w_i$. 
The {\em wedge product} $\wedge$  of twisted Hodge cohomology classes is defined by the composition of 
%\begin{multline*} %\label{eqn: wedge in tw}  
% \bigoplus _{p+q = m}  \HH ^{p} (X, (\Omega ^{\bullet} _X, dw_1)) \ot \HH ^{q} (X, (\Omega ^{\bullet} _X, dw_2)) \\
%         \xrightarrow{\text{\em K\"unneth}}         \HH ^{m} (X^2 , (\Omega ^{\bullet} _X, dw_1 ) \boxtimes (\Omega ^{\bullet} _X, dw_2 ))  \\
%                      \xrightarrow{\Delta ^*_X}         \HH ^{m} (X, (\Omega ^{\bullet} _X, dw_1 ) \otimes _{\cO_X}  (\Omega ^{\bullet} _X, dw_2 ))  \\
%                            \xrightarrow{ wedge }    \HH ^{m}_{Z_1\cap Z_2} (X, (\Omega ^{\bullet} _X, dw_1+ dw_2 ))  . \end{multline*}
\begin{align*} %\label{eqn: wedge in tw}  
&   \bigoplus _{q_1+q_2 = q \in \GG}   \HH ^{q_1} (X, (\Omega ^{\bullet} _X, dw_1)) \ot \HH ^{q_2} (X, (\Omega ^{\bullet} _X, dw_2)) \\
        & \qquad \qquad \xrightarrow{\text{\em K\"unneth}}      \quad   \HH ^{q} (X^2 , (\Omega ^{\bullet} _X, dw_1 ) \boxtimes (\Omega ^{\bullet} _X, dw_2 ))  \\
                    & \qquad \qquad \qquad   \qquad  \xrightarrow{\Delta ^*_X}      \quad     \HH ^{q} (X, (\Omega ^{\bullet} _X, dw_1 ) \otimes _{\cO_X}  (\Omega ^{\bullet} _X, dw_2 ))  \\
                         &   \qquad \qquad \qquad \qquad \qquad \xrightarrow{ wedge }   \quad   \HH ^{q}_{Z_1\cap Z_2} (X, (\Omega ^{\bullet} _X, dw_1+ dw_2 ))  . \end{align*}
Here $\Delta ^*_X$ is the pullback of the diagonal morphism $X \to X\ti X$. We sometimes omit the symbol $\wedge$ for the sake of simplicity. 
\end{Def}

\begin{Def}\label{def: ( )}  Let $Z$ denote the critical locus of $w$.
Consider a sequence of maps 
%\[ \begin{array}{lll}  \HH ^{*} (X, (\Omega ^{\bullet}_X, -dw))  \ti \HH ^{*} (X, (\Omega ^{\bullet}_X, dw))
%& \xrightarrow{\cdot \wedge \cdot } &  \HH ^{*} _Z(X, (\Omega ^{\bullet}_X, 0))  \\
%& \xrightarrow{\cdot \wedge \td (X) } &   \HH ^{*} _Z(X, (\Omega ^{\bullet}_X, 0))  \\
%& \xrightarrow{proj} & \HH ^0_Z (X, \Omega ^n _X [n] )  \\ 
%& \xrightarrow{natural}  & \HH ^0_c (X, \Omega ^n _X [n] ) \\
%& \xrightarrow{(-1)^{{n+1}\choose{2}} \int _X}  &  k . \end{array} \]
 \begin{multline}   \HH ^{*} (X, (\Omega ^{\bullet}_X, -dw))  \ti \HH ^{*} (X, (\Omega ^{\bullet}_X, dw))
 \xrightarrow{\cdot \wedge \cdot }   \HH ^{*} _Z(X, (\Omega ^{\bullet}_X, 0))  \\
 \xrightarrow{\wedge \td (X) }    \HH ^{*} _Z(X, (\Omega ^{\bullet}_X, 0))  
 \xrightarrow{proj}  \HH ^0_Z (X, \Omega ^n _X [n] )  
 \xrightarrow{}   \HH ^0_c (X, \Omega ^n _X [n] ) 
 \xrightarrow{(-1)^{{n+1}\choose{2}} \int _X}    k . \end{multline} 
Denote the composition  $ (-1)^{ {n+1}\choose{2} }  \int _X( \cdot \wedge \cdot \wedge \td (X))$ by $\lan , \ran$.
\end{Def}

\subsection{Proof of Theorem \ref{thm: exp pairing} }\label{pf exp pairing}
Since $\td (X)$ is invertible, the nondegeneracy of  $\lan , \ran $ follows from Serre's duality; see \cite[\S~4.1]{FK: GLSM}.
 Therefore it is enough to show that $\lan , \ran $ satisfies  \eqref{eqn: char pairing} under the HKR-type isomorphism in  \eqref{diag: big comm}.
Recalling \eqref{eqn: ch diag},
we write \[\ch (\cO ^{\widetilde{w}}_{\Delta _X}) = \sum_i t^i \ot t_i  \in  \bigoplus _{q \in \GG } \HH^{q} (X, (\Omega ^{\bullet}_X, dw)) \ot \HH ^{-q} (X, (\Omega ^{\bullet}_X, -dw )) .\]
For $\gamma \in \HH ^{*} (\Omega ^{\bullet}_{X}, -dw)$ and $\gamma ' \in \HH ^{*} (\Omega ^{\bullet}_{X}, dw) $, we have 
\begin{align}
  \sum _i \lan \gamma , t^i \ran \lan t_i, \gamma ' \ran 
& =    \int _{X\ti X}  ( \gamma \ot \gamma ') \wedge \ch (\cO ^{\widetilde{w}}_{\Delta _X}) \wedge (\td (X) \ot \td (X) ) ,  \label{eqn: ch Delta}
\end{align}
since $\int _X \ot _k \int _X = \int _{X \ti X} \circ \text{\it K\" unneth} $.

Since $\cO ^{\widetilde{w}}_{\Delta _X}$ is supported on the diagonal $\Delta _X \subset X \ti X$, we will apply the deformation of $X\ti X$ to 
the normal cone of $\Delta _X$. The normal cone is isomorphic to the tangent bundle $T_X$ of $X$. Let $\pi$ denote the projection $T_X \to X$.
%Note that \[\PP (T_X \oplus \cO _X) \supset \tot (T_X) = \PP (T_X \oplus \cO_X) - \PP (T_X) \supset \PP (\cO_X) = X \cong \Delta _X .\]
We claim a sequence of equalities
\begin{align*}
 \text{ RHS of }\eqref{eqn: ch Delta} 
& \overset{(\dagger )}{=}     \int _{T_X}  \pi^* (\gamma \wedge \gamma ') \wedge \ch (\mathrm{Kos} (s)) \wedge \pi^* \td (X) ^2  \\
%& \overset{(1)}{=}     \int _{\PP (T_X \oplus \cO_X )}  \pi^*( \gamma \wedge \ \gamma ' \wedge \td (X)) \wedge \ch (\mathrm{Kos} (s)) \wedge \td (Q)   \\
&  \overset{(\dagger\dagger)}{=}         (-1)^{\binom{n+1}{2}}  \int _{X} (\gamma \wedge \gamma ' \wedge \td (X)) 
 =   \lan \gamma, \gamma ' \ran ,
\end{align*} whose proof will be given below.
Here  $s$ is the `diagonal'  section of $\pi ^*T_X$ defined by $s(v) = (v, v) \in \pi^* T_X $ for $v\in T_X$ 
and $\mathrm{Kos} (s)$ is the Koszul complex $(\bigwedge ^{\bullet}  \pi^* T_X ^{\vee} , \iota _s )$ associated to $s$.

For $(\dagger)$ consider the deformation space $M^{\circ}$ of $X\ti X$ to the normal cone of the diagonal $\Delta_{X}$; see \cite{Fulton}.
%The space $M$ is constructed as the blow-up of $X\ti X \ti \PP ^1$ along the closed subscheme $\Delta _{X} \ti \infty$. 
It is a variety  with morphisms $h: M^{\circ} \to X\ti X$ and $pr: M^{\circ} \to \PP^1$, satisfying that (i)
the preimages of general points of $\PP^1$ are $X \ti X$, (ii) the preimage of a special point $\infty$ of $\PP ^1$ is 
the normal cone $N_{\Delta _X / X^2} = T_X$,
(iii) $pr$ is a flat morphism, and (iv) $h|_{T_{X}} $ coincides with the composition $\Delta \circ \pi$.

The morphism $\Delta \ti \mathrm{id}_{\mathbb{A}^1} : X  \ti \mathbb{A}^1  \to X \ti X \ti \mathbb{A}^1$ extends to a closed immersion $f: X \ti \PP ^1 \to M^{\circ}$.
For a closed point $p$ of $\PP^1$ let $M^{\circ}_p$ denote the fiber $pr ^{-1} (p)$ and consider the commuting diagram  
\[ \xymatrix{   X \ar[r]  \ar[d]_{ \Delta = f_0} &  X\ti \PP ^1 \ar[d]^{f} & \ar[l]  X \ar[d]^{f_{\infty} = \text{ zero section}}   \\
                                                          X^2 =     M^{\circ}_0  \ar[r]^{g_0} \ar[d] & M^{\circ}  \ar[rd]^{h} \ar[d]^{pr} &   \ar[l]_{g_{\infty}}   M^{\circ}_{\infty} = T_X  \ar[d]^{ \Delta \circ \pi  }   \\
                                                                0      \ar[r] &     \  \PP ^1 &  X^2    } \]
                                                             with three fiber squares.
 Since $X\ti \PP^1$ and $M^{\circ}_p$ are Tor independent over $M^{\circ}$,  we have
 \begin{equation}\label{eqn: Tor} \mathbb{L} g_p^* f_*  \cO _{X\ti \PP ^1} \underset{qiso}{\sim} (f_{p})_{ *} \cO _{X} , \end{equation}
 i.e., they are quasi-isomorphic as coherent factorizations for $(M_p, -h^* \widetilde{w} |_{M_p})$.                                      
Note that $h^*\widetilde{w}|_{\PP (T_{X} \oplus \cO_{X} )} = 0$. 
Since $s$ is a regular section with the zero locus  $X \subset T_X$,
two factorizations $(f_{\infty})_{ *}\cO _{X} $ and     $ \mathrm{Kos} (s)  $
 are quasi-isomorphic to each other  
as coherent factorizations for $( T _{X} , 0)$: 
\begin{equation}\label{eqn: Kos bar} (f_{\infty})_{ *}\cO _{X}      \underset{qiso}{\sim} \mathrm{Kos} (s) .   \end{equation} 
 For  $\rho =  ( \gamma \ot \gamma ') \wedge (\td (X) \ot \td (X)) $, we have a sequence of equalities
\begin{align*} & \int _{X \ti X} \rho   \wedge \ch (  (f_{0})_* \cO _{X}) & \\
 = &       \int _{X \ti X} \rho  \wedge \ch (  \mathbb{L}g_0^* f_* \cO _{X \ti \PP ^1} ))     &  \text{ by \eqref{eqn: Tor}}    \\
= & \int _{X \ti X}  g_0^*(h^* \rho  \wedge  \ch (  f_* \cO _{X \ti \PP ^1} )) & \text{ by the functoriality of $\ch$} \\
   = & \int _{T_{X}  }   g_{\infty}^*(h^* \rho  \wedge  \ch (  f_* \cO _{X \ti \PP ^1} )) &
    \text{ by Lemma \ref{lem: fiber base}} \\
 = &  \int _{T_{X} }   \pi ^* \Delta^* \rho  \wedge \ch (  \mathrm{Kos} (s)) &  \text{ by  \eqref{eqn: Tor} \& \eqref{eqn: Kos bar}} ,
 \end{align*}
 which shows  $(\dagger)$.

The equality $(\dagger\dagger)$ immediately follows from some basic properties of the proper pushforward \eqref{def: gen push} 
in Hodge cohomology: the functoriality \eqref{eqn: fun}, the projection formula \eqref{eqn: proj},
and \eqref{eqn: rel comp}.

%% \begin{equation}\label{eqn: pi com} 
%%\int _{\pi}   \ch (\mathrm{Kos} (s)) \wedge \td (\pi^* T_X)  = \int _{\pi}  c_n (Q) =  (-1)^{\binom{n+1}{2}}. \end{equation} Equality \eqref{eqn: pi com} holds by the deformation $s\to 0$, 
%%the base change \eqref{eqn: fiber base}, and  the computation \eqref{eqn: comp}.

\subsection{Proof of Theorem \ref{thm: main}}
For $\ka \in  \oplus _{i \in \GG} \RR^i \End (P)$ and $\kb \in \oplus _{i \in \GG} \RR^i\End (Q)$, let us define 
\[L_{\kb} \circ R_{\ka} :  \oplus _{i \in \GG}\RR ^i \Hom (P, Q) \to  \oplus _{i \in \GG} \RR ^i \Hom (P, Q) , \ c   \mapsto (-1)^{|\ka||c|} \kb\circ c\circ \ka .\]
Since $\ka$ and $\kb$ are cycle classes of $C (\cA)$, they can be considered as elements of $HH_* (\cA)$. 
We denote by $\tau (\ka)$, $\tau (\kb)$ be the image of $\ka$, $\kb$ under the HKR map. The map $\tau$ is sometimes called the boundary-bulk map. 
Combining \eqref{eqn: abs general HRR} and Theorem \ref{thm: exp pairing} we obtain this.
\begin{Cor}\label{cor: gen} (The Cardy Condition) We have
\begin{equation}\label{eqn: general HRR}  \str (L_{\kb} \circ R_{\ka}  )     = (-1)^{ \binom{n+1}{2} }  \int _X  \tau (\kb ) \wedge \tau (\ka) ^{\vee} \wedge \td (X)   .\end{equation} 
In particular, Theorem \ref{thm: main} holds.
\end{Cor}

Corollary \eqref{eqn: general HRR} is the matrix factorization version
of Theorem 16 of \cite{CW} and the explicit Cardy condition in  \cite{Ram: HRR}.

Let $\fU = \{ U_i \}_{i\in I} $ be an affine open covering of $X$ and let $\nabla _i$ be a connection of $P|_{U_i}$, which always exists. 
By \cite{CKK,  KP, Platt} the following formula for $\tau (\ka)$ in the \v{C}ech cohomology $\check{\HH}^0 (\fU, (\Omega ^{\bullet}_X, dw )$
is known:
\[ \tau (\ka) =  \str \left( \big( \exp (-([\nabla _i, \delta _E])_i - (\nabla _i - \nabla _j)_{i < j} ) \big)\check{\ka}  \right)  , \]
where $\check{\ka}$ is a \v{C}ech representative of $\ka$. 
Here we recall that  $\Omega ^{\bullet}_X = \oplus _{p=0}^n \Omega _X ^p [p] $ is $\GG$-graded.

In the local case, i.e., $X$ is an open neighborhood of the origin $0$ in $\mathbb{A} ^n_k$
and $w$ has a critical point only at $0$ with $w(0)=0$, we can relate the canonical pairing  with a residue pairing. 
Let $x=(x_1, ..., x_n)$ be a local coordinate system 
and let $\partial _i w = \frac{\partial w}{\partial x_i}$.
Proposition 4.34 of \cite{BW: ShkConj} shows that 
\[ \int _X ( \tau (\kb) \wedge \tau (\ka) ^{\vee}   ) = \underset{x=0}{\mathrm{Res}} \left[ \frac{g(x) f(x) }{\partial _1 w, ..., \partial _n w} \right] \] 
for $\tau (\ka) = f(x) dx_1 ... dx_n$, $\tau (\kb) = g(x) dx_1 ... x_n$ in $\Omega ^n _{X} / dw \wedge \Omega ^{n-1}_X$.
Hence from Theorem \ref{thm: exp pairing} and $\tau (\ka ^{\vee}) = \tau (\ka )^{\vee}$ we immediately obtain this.

\begin{Cor} \cite{BW: ShkConj, PV: HRR}
In the local case we have \[ \lan \kb , \ka ^{\vee} \ran _{can} = (-1)^{ \binom{n+1}{2}}   \underset{x=0}{\mathrm{Res}} \left[ \frac{g(x) f(x)}{\partial _1 w, ..., \partial _n w} \right]  . \]
\end{Cor}

The corollary above reproves a conjecture of Shklyarov \cite[Conjecture 3]{Shk: Residue}.

\subsection{GRR type theorem}\label{sub: GRR}
Consider the proper morphism $f: X \to X'$ in \S~\ref{sub: geom diag},  inducing the dg functor $\RR f_*$ and  the module $\Delta _{\RR f_*} \in \Perf (\cA ^{op} \ot \cB)$.
They together  make a commutative diagram
\begin{equation*} \xymatrix{  \ar[d]_{Yoneda} \cA := \MFdg (X, w) \ar[r]^{  \RR f_* }  &  \cB := \MFdg (X' , w') \ar[d]^{Yoneda} \\
                                                        \Perf \cA    \ar[r]_{T_{\Delta _{\RR f_*}}} & \Perf \cB  . } \end{equation*}

The paring defined by the composition 
\[ \HH ^{*} (X', (\Omega ^{\bullet}_{X'}, -dw')) \ot  \HH ^{*} (X', (\Omega ^{\bullet}_{X'}, dw')) \xrightarrow{\wedge} \HH ^{*} _c (X', (\Omega ^{\bullet}_{X'}, 0)) \xrightarrow{\int _{X'}} k \]
is nondegenerate by the Serre duality; see \cite[\S~4.1]{FK: GLSM}.
Using the paring we define the pushforward for $q\in \GG$
\[ \int _f :   \HH ^{q} (X, (\Omega ^{\bullet}_{X}, -dw)) \to  \HH ^{q} (X', (\Omega ^{\bullet}_{X'}, -dw')) \]
by the projection formula requirement
\[ \int _{X'} ( \int _f \ka ) \wedge \beta  = \int _{X} \ka \wedge f^* \beta  \]
for every  $\beta \in   \HH ^{-q} (X', (\Omega ^{\bullet}_{X'}, dw'))$.

Let $n= \dim X$ and $m=  \dim X'$. Denote by $HH (\RR f_* )$ the map in Hochschild homology level from $\RR f_*$.
Let $K_0 (\cA)$, $K_0 (\cB)$ be the Grothendieck group of the homotopy category of $\cA$, $\cB$, respectively.
\begin{Thm} 
The diagram 
\begin{equation*} \xymatrix{ K_0 (\cA ) \ar[d]_{\Ch} \ar[rr]^{\RR f_*} & & K_0 (\cB) \ar[d]^{\Ch} \\ 
\ar[d]_{I_{HKR}}  HH_* (\cA) \ar[rr]^{HH(\RR f_*)}  & & HH_* (\cB )  \ar[d]^{I_{HKR}}  \\
                                            \HH ^{-*} (X, (\Omega ^{\bullet}_{X}, -dw)) \ar[rr]_{ (-1)^{\sharp } \int _{f} \cdot \wedge \td (T_f) } & &  \HH ^{-*} (X', (\Omega ^{\bullet}_{X'}, -dw')) } \end{equation*}
is commutative.  Here $\td (T_f) : = \td (X) / f^* \td (X')$ and $\sharp = \binom{n+1}{2} -  \binom{m+1}{2} $.
\end{Thm}

\begin{proof} By the definition of categorical Chern characters the upper rectangle is commutative. 
Consider  $\gamma  \in HH_* (\cA) $. Let $\ka := I_{HKR} (\gamma )$ and  $\ka ' :=  I_{HKR} (HH (\RR f_* ) (\gamma )) $. 
If we write $\ch (\Delta _{\RR f_*})   = \sum_i T^i \ot T_i \in   \HH ^{*} (X, (\Omega ^{\bullet}_{X}, dw))  \ot   \HH ^{*} (X', (\Omega ^{\bullet}_{X'}, -dw)) $,
then by  Proposition \ref{prop: char can} and Theorem \ref{thm: exp pairing} we have 
for  $\beta \in   \HH ^{-*} (X', (\Omega ^{\bullet}_{X'}, dw'))$
\begin{equation}\label{eqn: Rf 1}
 \int _{X'}  \ka '   \wedge \beta \wedge \td (X') = (-1)^{{n+1}\choose{2}} \sum _i  \int _X \ka \wedge T^i \wedge \td (X) \int _{X'} T_i \wedge \beta \wedge \td (X') . \end{equation}
By \eqref{eqn: ch Rf} and a normal-cone deformation argument as in \S~\ref{pf exp pairing} we have
\begin{align*} &  \mathrm{RHS} \text{ of } \eqref{eqn: Rf 1}  \\
& = (-1)^{{n+1}\choose{2}}  \int _{X\ti X' }  (\ka   \ot \beta) \wedge \ch (\cO ^{\widetilde{w}}_{\Gamma _f }) \wedge (\td (X) \ot \td (X')) \\
& = (-1)^{{n+1}\choose{2}}  \int _{f^*T_{X'}} \pi ^* (\ka \wedge f^* \beta  \wedge \td (f^* T_{X'}) \wedge  \td (X))  \wedge \ch (\mathrm{Kos} (s))  \\
& = (-1)^{\sharp }  \int _{X} \ka \wedge f^* \beta \wedge \td (X)  = (-1)^{\sharp }  \int _{X'} (\int _f \ka \wedge \td (X) ) \wedge \beta ,\end{align*} 
where $\pi$ denotes the projection $f^*T_{X'}  \to X$ and  $s$ is the diagonal section
of $\pi^*f^*T_{X'}$ on $f^*T_{X'} $.
Hence $\mathrm{LHS}\text{ of } \eqref{eqn: Rf 1}$ equals  $ (-1)^{\sharp}  \int _{X'} (\int _f \ka \wedge \td (X) ) \wedge \beta $,  which shows the commutativity of the lower rectangle.
\end{proof}

\subsection{Pushforward in Hodge cohomology}
We collect some properties of pushforwards in  Hodge cohomology that are used in \S~\ref{pf exp pairing}. For lack of a suitable reference we provide their proofs.

Throughout this subsection  $f : X \to Y$ will be a morphism  between varieties $X$, $Y$ with dimensions $n$, $m$ respectively.
Let $d= n - m$.

\subsubsection{Definition of $f_*$} 
Suppose that $f$ is a proper locally complete intersection (l.c.i) morphism. 
Let $E$ be a perfect complex on $Y$. Denote  by  \[ \tau _f : \RR f_* f^!  E \to E \] 
 the duality map in the derived category $D^+_{qc} (\cO_Y)$ of cohomologically bounded below quasi-coherent sheaves; see 
 for example  \cite[\S~4]{Lipman: Found}.
Since $f$ is l.c.i,  $f^!  \cO_Y$ is taken to be an invertible sheaf up to shift and there is 
 a canonical isomorphism $\LL f^*E \ot f^!\cO _Y \cong f^! E $. 
 For $q\in \ZZ$, let $\delta \in \HH ^q  (X, f^! E)$, which can be considered as a map $\delta : \cO _X [-q] \to f^! E $ in the derived category.
We have a composition of maps 
\[ \cO _Y [-q] \xrightarrow{natural}  \RR f_* f^* \cO _Y [-q] \xrightarrow{\RR f_* (\delta ) } \RR f_* f^! E  \xrightarrow{ \tau _f } E  , \] 
denoted by $f_* (\delta )$. This yields a homomorphism 
\[ f_* :   \HH^q (X, f^! E) \to \HH ^q  (Y , E) . \]

Let $g: Y \to Z$  be a proper l.c.i. morphism between varieties. 
The uniqueness of adjunction implies the functoriality of the pushforward 
\begin{equation}\label{eqn: gen fun}  ( g \circ f )_*  =    g_*    \circ f_*:  \HH ^q (X,  ( g \circ f )^! F ) \to \HH ^{q} (Z, F )  \end{equation}
for $F$ in $D^+_{qc} (\cO_Z)$

\subsubsection{Definitions of $\int _{f}$ and $\int _X$}\label{def int}
Let $f : X \to Y$ be a morphism between nonsingular varieties. 
For $p\ge 0$ with $p-d \ge 0$ we have a natural homomorphism 
\begin{multline*} 
  \Omega ^{p}_{X} [q]  \cong  \bigwedge^{n - p } T_X [ q] \ot f^* \Omega _Y^m [-d] \ot f^! \cO _Y  \\
 \xrightarrow{}  \bigwedge^{n - p } f^* T_Y [ q] \ot f^* \Omega _Y^m [-d]  \ot f^! \cO _Y  \cong  f^* \Omega _Y^{p-d} [q-d]  \ot f^! \cO _Y  .
  \end{multline*} denoted by  $\mathscr{T} _f.$ 

We define {\em Hodge cohomology with proper supports along $f$} as the direct limit:
 \[ H^q_{cf} (X, \Omega ^{p} _{X}) := \lim_{\longrightarrow} H^q_Z (X,  \Omega ^{p} _{X}) , \]
where $Z$ runs over all closed subvarieties of $X$ that are proper over $Y$.
By Nagata's compactification and the resolution of singularities
there is a nonsingular variety $\bar{X}$ including $X$ as an open subvariety 
and a proper morphism $\bar{f}: \bar{X} \to Y$ extending $f$. 
Recall the fact that if $Z$ is a closed subvariety of $X$ that is proper over $Y$, then $Z$ is a closed subvariety of $\bar{X}$. Let
   \[ nat:  H^q_{\natural _1} (X, \Omega ^{p} _{X}) \to H^q_{\natural _2} (\bar{X}, \Omega ^{p}_{\bar{X}})   \]  be the natural map where $(\natural _1, \natural _2)$ is either $(c, c)$
or $(cf, \emptyset)$.
We define the pushforward (for $p\ge 0$ with $p-d \ge 0$)
\begin{equation}\label{def: gen push} \int _{f}   : H^q _{\natural _1} (X, \Omega ^{p} _{X}) \to H^{q-d}_{\natural _2} (Y, \Omega ^{p-d} _Y ) ; \gamma \mapsto \bar{f}_* ( \mathscr{T}_f (nat (\gamma ))) . \end{equation}
 Using the functoriality \eqref{eqn: gen fun}, we note that  $\int _{f} $ is independent of the choices of $\bar{X}$, an open immersion $X\hookrightarrow \bar{X}$, and an extension $\bar{f}$.
When $Y=\Spec k$, we also write $\int _{X}$ for $\int _{f}$

If $ v : X' \to X$ be a proper morphism between nonsingular varieties, we have the  natural pullback map
\[ v^* : H^q_c (X, \Omega ^p_{X}) \to H^q_c (X' , \Omega ^p_{X'}) . \]

\subsubsection{Base change I}
Consider a fiber square diagram of varieties 
\begin{equation}\label{diag: fiber diag} \xymatrix{ X' \ar[r]^{v} \ar[d]_g & X \ar[d]^f \\
                       Y' \ar[r]_u & Y .  } \end{equation} 
                       Assume that $f$ is a flat, proper, l.c.i morphism.
                       Then from the base change \cite[\S~4.4]{Lipman: Found} we 
                       obtain a base change formula,   for  $\delta \in H^q (X, f^! \cO_Y ) $
\begin{align}\label{eqn: base}  g_*   (\LL v^* (\delta )) =   \LL u^* ( f_* ( \delta ) )   \end{align}
in $H^q (Y', \cO _{Y'})$.
   Here $\LL v^*(\delta ) \in  H^ q (X' , g^! \cO _Y) $  is the naturally induced map  
   \[ \cO_{X'} [-q] \to  \LL v^* f^! \cO_Y   \cong g^! \LL u^* \cO_Y  = g^! \cO _{Y'} \] in the derived category.

Furthermore suppose that  all varieties $X, Y, X'$ are nonsingular and $Y'$ is a closed point of $Y$. Then for $\gamma \in H^d _{c} (X, \Omega ^d _{X})$ we easily check that
\[  v^* (\gamma ) = \LL v^* (\mathscr{T}_f (\gamma ))   \] in  $H^d _{c} (X', \Omega ^d _{X'}) = H^0 _{c} (X', g^! \cO_{Y'}) $.
Hence \eqref{eqn: base} for $\delta = \mathscr{T}_f (\gamma)$ means that
   \begin{align}\label{eqn: smooth base}  \int _{X'}    v^* (\gamma) =   u^* (\int _f \gamma )  . \end{align}
%If $Y$ is complete, then $H^0 (Y, \cO_Y) = k$

    \subsubsection{Base change II} 
  Let  $Y$ be a connected nonsingular  complete curve and let $Y'$ be  a closed point of $Y$.
Consider the fiber square diagram \eqref{diag: fiber diag} of nonsingular varieties. 
Assume that $f$ is flat but possibly non-proper. 
          
   \begin{lemma} \label{lem: fiber base}
  For $ \gamma \in H^d _c (X, \Omega ^d _X)$ we have
  \begin{equation}\label{eqn: fiber base}  \int _{X'} v ^* (\gamma )   =  \int _{f} \gamma \quad \in k .  \end{equation}
   \end{lemma} 
\begin{proof} 
By Nagata's compactification  $f$ is extendible to a proper flat morphism
                       $\bar{f}: \bar{X} \to Y$ with an open immersion $X\hookrightarrow \bar{X}$. By the resolution of singularities we can make that $\bar{X}$ is nonsingular 
                     and  the closure $\bar{X'}$ of $X'$ in $\bar{X}$ is also nonsingular. Let $\bar{v} : \bar{f}^{-1} (Y') \to \bar{X}$ be the induced morphism and let  $\bar{v}_{\circ} := \bar{v} |_{\bar{X'}} $.
                    Thus we have a commutative diagram
                     \begin{equation}\label{diag: comp fiber diag} \xymatrix{ \bar{X'}\  \ar@{^{(}->}[r] \ar[rd]_{\bar{g}_{\circ} } \ar@/^1.3pc/[rr]^{\bar{v}_{\circ}}  
                     & \bar{f}^{-1}(Y') \ar[r]_(.6){\bar{v}} \ar[d]_{\bar{g}} & \bar{X} \ar[d]^{\bar{f}} \\
                     &   Y' \ar[r]_u & Y ,  } \end{equation} with a fiber square. 
To show \eqref{eqn: fiber base} we may assume  $ \gamma \in H^d _Z (X, \Omega ^d _X)$ for some complete subvariety $Z$ of $X$.
Let $nat$ denote the natural map $H^d _Z (X, \Omega _X^d) \to H^d (\bar{X} , \Omega _{\bar{X}}^d)$. Then by the support condition of $\gamma$ 
we have
\begin{equation}\label{eqn: LH RH}
 (\bar{g}_{\circ})_*  \LL \bar{v}_{\circ}^* (\mathscr{T}_{\bar{f}} ( nat (\gamma ) )) 
= \bar{g}_*  \LL \bar{v}^* (\mathscr{T}_{\bar{f}} ( nat (\gamma ) ))   \in k . \end{equation}
Since $ \bar{v}_{\circ}^* ( nat ( \gamma )) =  \LL \bar{v}_{\circ}^* (\mathscr{T}_{\bar{f}} (nat(\gamma )) )  $ under $\Omega ^d _{\bar{X'}} \cong g^! \cO _{Y'}$,
LHS of  \eqref{eqn: LH RH} becomes $\int _{\bar{X'}} \bar{v}_{\circ}^* ( nat ( \gamma )) $, which equals to LHS of \eqref{eqn: fiber base}  by the support condition of $\gamma$.
On the other hand by \eqref{eqn: base}, RHS of \eqref{eqn: LH RH} becomes  $u^* \int _{\bar{f}} nat (\gamma) $, which equals to RHS of \eqref{eqn: fiber base} by the support condition and
$H^0 (Y, \cO_Y) = k$.
\end{proof}

\subsubsection{Projection formula} 
Let $X$, $Y$, $Z$  be nonsingular varieties and let $f: X \to Y$, $g: Y\to Z$  be morphisms. Let $d' = \dim Y - \dim Z$.
The uniqueness of adjunction implies the functoriality of the pushforward, for $p\ge 0$ with $p-d \ge 0$ and $p-d - d' \ge 0$
\begin{equation}\label{eqn: fun} \int_{g\circ f}  = \int _g   \circ \int _f  :  H^q_c (X, \Omega ^{p} _X) \to H^{q-d- d'} _c (Z, \Omega ^{p-d-d'} _Z ) . \end{equation}

Let $f:X \to Y$ be a (possibly non-proper) morphism between nonsingular varieties. 
Then for  $\gamma \in H^d _{cf} (X, \Omega ^d _{X})$ and $\sigma  \in H^{q} (Y, \Omega ^{p} _{Y})$
the projection formula
\begin{align}\label{eqn: proj}  \int _f (  f^* \sigma \wedge \gamma )  = \sigma  \wedge \int _f  \gamma     \end{align}
holds in $H^q (Y, \Omega ^p _{Y} )$. This can be verified as follows. We may assume that $f$ is proper. Consider the commuting diagram
\[ \xymatrix{ \Omega _{Y} ^{p}   \ar[r]  % \ar@/^3pc/[rrrr] 
& \RR f_* \Omega _{X} ^{p}   \ar[rr]^{\RR f_* (\cdot \wedge \gamma ) \ \ \ \  \  }  & 
& \RR f_* ( \Omega _{X} ^{p}  \wedge \Omega ^{d}_{X} [d])  \ar[rr]^{\RR f_*(\mathscr{T}_f ) }     &        &        \Omega ^{p}_{Y} \ot \RR f_* f^!\cO_Y   \\
 \cO_Y [-q] \ar[u]^{\sigma} \ar[r] & \RR f_* \cO _X \ar[u]^{\RR f_* f^* \sigma } \ar@/^.7pc/[urr]_(0.4){\ \ \ \RR f_* (f^* \sigma \wedge \gamma)}  
       \ar@/_1pc/[rrrru]_{\ \ \RR f_* (\mathscr{T}_f (f^*\sigma \wedge \gamma))}
 & & &  &   }    \] We note that the composition of the maps in the top horizontal line is $\mathrm{id}_{\Omega ^p _Y} \ot \RR f_* ( \mathscr{T}_f ( \gamma))$
 using the generic smoothness of $f$ and local coordinate systems for compatible bases of $\Omega ^1 _{X}$ and  $\Omega ^1_Y$.
 The clockwise compositions of maps starting from $\cO _Y[-q]$ followed by $\tau _f$ yields LHS of \eqref{eqn: proj} and 
 the counterclockwise compositions of maps followed by $\tau _f$ yields RHS of \eqref{eqn: proj}.

\subsubsection{Some computations}

Let  $Q$ be the tautological quotient bundle on the projective space $\PP^n$. We want to compute $\int _{\PP^n}$ of the top Chern class $c_{n} (Q) \in H^0 (\PP ^n, \Omega ^n _X [n ] )$.   
The class $c_{n} (Q)$  is equal to $(-1)^nc_1(\cO (-1))^n$. Let $U_i = \{ x_i \ne 0 \}$ where $x_0, ..., x_n$ are homogeneous coordinates. 
On each $U_i$, we may identify $\cO (-1)$ with the $i$-th component of $\cO _{\PP^n} ^{\oplus n+1}$ by the tautological monomorphism 
$\cO (-1) \to \cO _{\PP^n} ^{\oplus n+1}$. This yields connections $\nabla _i$ on $\cO (-1) |_{U_i}$.
Let $z_i = x_i / x_0$. 
Note that $\nabla _0 - \nabla _i = - \frac{dz_i}{z_i}$.  
Hence $\nabla _i - \nabla _j =   \frac{dz_i}{z_i}  - \frac{dz_j}{z_j}$ on %$U_{0,i,j} := 
$U_0 \cap U_i \cap U_j$.
By the $n$-th fold  Alexander-\v{C}ech-Whitney cup product of a \v{C}ech representative $(\nabla _i - \nabla _j ) _{i<j}$ of $c_1(\cO (-1))$ we conclude that 
$ c_n (Q) $ is representable by a \v{C}ech cycle
\begin{equation*}   (-1)^{ \binom{n+1}{2}}   \frac{dz_1 ...  dz_n}{z_1 ... z_n}  \in  \Omega ^n _{\PP^n} (U _0 \cap ... \cap U_n) . \end{equation*}
Here the sign contribution of $ \binom{n}{2}$ among $ \binom{n+1}{2}$ 
comes from the exchanges of odd \v{C}ech `elements' and differential one forms $\frac{dz_i}{z_i}$; see  \cite{BW: ShkConj,  CKK}.
Thus \begin{equation}\label{eqn: comp} \int _{\PP ^n_k} c_n (Q) =   (-1)^{ \binom{n+1}{2}} \mathrm{res}  [ \frac{dz_1 ...  dz_n}{z_1  ...  z_n} ] = (-1)^{ \binom{n+1}{2}} . \end{equation}

Let $E$ be a rank $n$ vector bundle on a nonsingular variety $X$ and let $\pi: E\to X$ be the projection. 
We have the diagonal section $s$ of $\pi ^* E$ by letting $s(e) = (e, e)$. 
 Let $\bar{\pi} : \PP (E \oplus \cO_X ) \to X$ be the projection, which is a proper extension of $\pi$:
 \[ \PP (E \oplus \cO_X )  = \PP (E) \sqcup E \supset E \supset \PP (\cO_X) = X .  \]
Let $\cQ$ be the tautological quotient bundle on $\PP (E \oplus \cO_X )$.
It has a section $\bar{s}$ by the composition $\cO \xrightarrow{(0, -\mathrm{id})}  \pi ^*E \oplus \cO \xrightarrow{quot} \cQ$. 
Note that the zero locus $\bar{s}$ is $\PP (\cO_X)$, since $(0, - \mathrm{id})$ is factored through the kernel of $quot$ exactly on $\PP (\cO _X)$. 
Note that the composition $quot \circ (\mathrm{id}, 0) |_{E} : \pi ^* E \to \cQ |_{E}$ is an isomorphism sending $s$ to $\bar{s}|_E$. Therefore we have
\begin{multline}\label{eqn: rel comp}
  \int _{\pi} \ch (\mathrm{Kos} (s)) \td (\pi^* E)    
 = \int _{\bar{\pi}} \ch (\mathrm{Kos} (\bar{s})) \td (\cQ) \\
 = \int _{\bar{\pi}} c_{n} (\cQ) \text{ (by letting $\bar{s}=0$)} 
 =    (-1)^{ \binom{n+1}{2}} \text{ (by \eqref{eqn: smooth base}  \& \eqref{eqn: comp})}. \end{multline}

\end{document}